\documentclass[11pt]{amsart}

\pdfoutput=1

\usepackage[text={400pt,660pt},centering]{geometry}

\usepackage{amsmath,amssymb}
\usepackage{esint,amssymb} 
\usepackage{graphicx}
\usepackage{MnSymbol}
\usepackage{mathtools} 
\usepackage[colorlinks=true, pdfstartview=FitV, linkcolor=blue, citecolor=blue, urlcolor=blue,pagebackref=false]{hyperref}
\usepackage{microtype}

\usepackage{bm}
\usepackage{dsfont}

\newtheorem{proposition}{Proposition}
\newtheorem{theorem}[proposition]{Theorem}

\newtheorem{corollary}[proposition]{Corollary}

\theoremstyle{remark}

\theoremstyle{definition}

\numberwithin{equation}{section}
\numberwithin{proposition}{section}
\numberwithin{figure}{section}
\numberwithin{table}{section}

\newcommand{\N}{\mathbb{N}}

\newcommand{\R}{\mathbb{R}}

\newcommand{\E}{\mathbb{E}}

\newcommand{\ep}{\varepsilon}
\newcommand{\eps}{\varepsilon}

\renewcommand{\le}{\leqslant}
\renewcommand{\ge}{\geqslant}
\renewcommand{\leq}{\leqslant}
\renewcommand{\geq}{\geqslant}

\newcommand{\la}{\left\langle}
\newcommand{\ra}{\right\rangle}

\newcommand{\Ll}{\left}
\newcommand{\Rr}{\right}
\renewcommand{\d}{\mathrm{d}}

\DeclareMathOperator{\supp}{supp}

\renewcommand{\tilde}{\widetilde}

\renewcommand{\hat}{\widehat}

\newcommand{\1}{\mathds{1}}

\newcommand{\mcl}{\mathcal}

\newcommand{\mfk}{\mathfrak}
\newcommand{\al}{\alpha}

\newcommand{\de}{\delta}
\newcommand{\si}{\sigma}

\newcommand{\dr}{\partial}

\newcommand{\FF}{\mathsf{F}}

\newcommand{\Pnl}{\Phi_{\nu,\lambda}}

\newcommand{\err}{{\mathsf{err}_N}}
\renewcommand{\b}{\mathbf{b}}

\makeatletter
\newsavebox\myboxA
\newsavebox\myboxB
\newlength\mylenA
\newcommand*\mybar[2][0.75]{%
    \sbox{\myboxA}{$\m@th#2$}%
    \setbox\myboxB\null
    \ht\myboxB=\ht\myboxA%
    \dp\myboxB=\dp\myboxA%
    \wd\myboxB=#1\wd\myboxA
    \sbox\myboxB{$\m@th\overline{\copy\myboxB}$}
    \setlength\mylenA{\the\wd\myboxA}
    \addtolength\mylenA{-\the\wd\myboxB}%
    \ifdim\wd\myboxB<\wd\myboxA%
       \rlap{\hskip 0.5\mylenA\usebox\myboxB}{\usebox\myboxA}%
    \else
        \hskip -0.5\mylenA\rlap{\usebox\myboxA}{\hskip 0.5\mylenA\usebox\myboxB}%
    \fi}
\makeatother

\newcommand{\rpc}{\hspace{0.2mm}\d\mathfrak{R}}
\newcommand{\zdm}{{\zeta_\mu}}

\newcommand{\e}{{\mathbb{E}}}

\newcommand{\Reals}{\mathbb{R}}

\newcommand{\DD}{\mathcal{D}}

\newcommand{\LL}{\mathcal{L}}
\newcommand{\MM}{\mathcal{M}}

\newcommand{\PP}{\mathcal{P}}

\newcommand{\bigO}[1]{\ensuremath{\mathop{}\mathopen{}\mathcal{O}\mathopen{}\left(#1\right)}}

\begin{document}

\author[J.-C. Mourrat]{J.-C. Mourrat}
\address[J.-C. Mourrat]{Courant Institute of Mathematical Sciences, New York University, New York, New York, USA}
\email{jcm777@nyu.edu}

\author[D. Panchenko]{D. Panchenko}
\address[D. Panchenko]{Department of Mathematics, University of Toronto, Toronto, Ontario, Canada}
\email{panchenk@math.toronto.edu}

\date{\today}

\title{Extending the Parisi formula along a Hamilton-Jacobi equation}

\begin{abstract}
We study the free energy of mixed $p$-spin spin glass models enriched with an additional magnetic field given by the canonical Gaussian field associated with a Ruelle probability cascade. We prove the conjecture in \cite{Mourrat19} that this free energy converges to the Hopf-Lax solution of a certain Hamilton-Jacobi equation. Using this result, we give a new representation of the free energy of mixed $p$-spin models with soft spins.
\end{abstract}

\maketitle

%
%
%
%
%
%

\section{Introduction}
\label{s.intro}

Let $(\beta_p)_{p \ge 2}$ be a sequence of real numbers and let $\xi(r) := \sum_{p \ge 2} \beta_p^2 r^p$ for $r \in \R$. We will assume that the sequence $(\beta_p)_{p \ge 2}$ is such that $\xi$ is well defined on the entire real line. This assumption can be relaxed if needed by restricting the parameters of the models we will be working with.
Denote by $(H_N(\sigma))_{\sigma \in \R^N}$ the centered Gaussian field with covariance
\begin{equation*}  
\E \Ll[ H_N(\sigma) \, H_N(\tau) \Rr] = N \xi \Bigl( \frac{\sigma \cdot \tau}{N} \Bigr), \qquad \sigma,\tau \in \R^N.
\end{equation*}
Let $P_N := P_1^{\otimes N}$ denote the $N$-fold product of $P_1$, a probability measure on $\R$ with bounded support. We aim to study the Gibbs measure built with respect to the energy function $H_N(\sigma)$ and the reference measure $P_N$. A quantity of fundamental interest is the limit free energy
\begin{equation}  
\label{e.lim.free.energy}
\lim_{N \to \infty} \frac 1 N \E \log  \int\! \exp \Ll( H_N(\sigma) \Rr)\, \d P_N(\sigma).
\end{equation}
When the support of $P_1$ is $\{\pm 1\}$ and $\xi(r)=\beta^2 r^2$ (the Sherrington-Kirkpatrick model~\cite{SK}), this limit was  discovered by Parisi in a celebrated work \cite{Parisi79, Parisi}; see also~\cite{MPV}. The formula was then proved rigorously for general $\xi$ in \cite{Guerra, TPF, PanPF}, and was later extended to the current setting where we only assume that the support of $P_1$ is bounded in \cite{Pan05, SKvector}.

\smallskip

In order to further our understanding of this object, it was proposed in \cite{Mourrat19} (following \cite{HJinfer, HJrank}) to recast the limit free energy \eqref{e.lim.free.energy} as a particular value of the solution of a Hamilton-Jacobi equation. This solution depends on two parameters $t \ge 0$ and $\mu \in \MM(\R_+)$, where $\MM(\R_+)$ denotes the set of Borel probability measures over $\R_+$. It was conjectured that an enriched version of the free energy, which would depend additionally on the parameters $t \ge 0$ and $\mu \in \MM(\R_+)$, may converge to the same solution evaluated at these parameters. 

\smallskip

The main purpose of this paper is to prove this conjecture. In order to state the result, we start by defining the enriched model precisely. We denote by $\MM_b(\R_+)$ the subset of $\MM(\R_+)$ of measures with bounded support. By \cite[Theorem~2.17]{SKmodel}, one can associate a Ruelle probability cascade \cite{RPC} to each probability measure on $[0,1]$; this Ruelle probability cascade is a random probability measure on the unit ball of a Hilbert space. We denote by $\mathfrak{R}$ the Ruelle probability cascade corresponding to the uniform distribution over $[0,1]$, and by $(\alpha^\ell)_{\ell\geq 1}$ an i.i.d.\ sample from~$\mfk R$ (that is, the law of $(\alpha^\ell)_{\ell\geq 1}$ is $\mfk R^{\otimes \infty}$). In particular, the law of the overlap $\alpha^1\cdot\alpha^2$ under $\e \mfk R^{\otimes 2}$ is the uniform distribution over $[0,1]$. Given a measure $\mu\in \MM_b(\R_+)$, and conditionally on $\mfk R$, let $z^\mu(\alpha)$ be a Gaussian process indexed by $\alpha\in\supp(\mfk R)$ with  covariance 
$$
\e \Ll[ z^\mu(\alpha^1)z^\mu(\alpha^2) \Rr]= \mu^{-1}(\alpha^1\cdot\alpha^2), \qquad  \al^1, \al^2 \in \supp(\mfk R).
$$ 
In the expression above and throughout the paper, we use the shorthand notation, for every $r \in [0,1]$, 
\begin{equation}  
\label{e.def.mu-1}
\mu^{-1}(r) := \inf \{s \ge 0 \ : \ \mu([0,s]) \ge r\}.
\end{equation}
To check that the Gaussian process $z^\mu(\al)$ exists, it suffices to verify   that $\mu^{-1}(\alpha^1\cdot\alpha^2)$ is a positive semidefinite kernel on $(\supp \mfk R)^2$, and this follows from the fact that the support of $\mfk R$ is ultrametric. Moreover, $\mu^{-1}(r)$ is left-continuous and, thus, continuous at $r=1$, which implies that the process $z^\mu(\alpha)$ is stochastically continuous on $\supp(\mfk R)$. As a result, it is jointly measurable (see e.g.\ \cite[Theorem 3.3.1]{Skor}) and we can define, for every $t \ge 0$ and $\mu \in \MM_b(\R_+)$, 
\begin{multline}  \label{FEth}
  F_N(t,\mu) := \frac 1 N \E \log \iint  \!\exp \Big( \sqrt{t} H_N(\sigma) + \sum_{i=1}^{N}\sigma_i z_i^\mu(\alpha)
\\
- \frac{1}{2}Nt\xi(N^{-1} |\si|^2) - \frac{1}{2}\mu^{-1}(1)|\si|^2 \Big)  \, \d P_N(\sigma) \rpc(\alpha),
\end{multline}
where $z_i^\mu(\alpha)$ are independent copies of $z^\mu(\alpha)$ for $i\geq 1$ (conditionally on $\mfk R$ and independent of $H_N$). 
For a measure of the form
\begin{equation}  
\label{e.music}
\mu = \sum_{\ell=0}^{k} (\zeta_{\ell+1}-\zeta_\ell)\delta_{q_\ell}
\end{equation}
with
\begin{equation}
0=\zeta_0<\zeta_1<\ldots<\zeta_k<\zeta_{k+1}=1, \qquad 
0\leq q_0<q_1<\ldots<q_k<\infty,
\label{music}
\end{equation}
one can rewrite $F_N(t,\mu)$ in the more familiar form
\begin{multline} \label{FE}
  F_N(t,\mu) = \frac 1 N \E \log \int_{\R^N} \sum_{\al \in \N^k} \exp \Big( \sqrt{t} H_N(\sigma) + \sum_{i=1}^{N}\sigma_iz_{\alpha,i}
\\
- \frac{1}{2}Nt\xi(N^{-1} |\si|^2) - \frac{1}{2}\mu^{-1}(1)|\si|^2 \Big)\, v_\al \, \d P_N(\sigma),
\end{multline}
where $(v_\alpha)_{\al \in \N^k}$ are the weights of the Ruelle probability cascade  with parameters $(\zeta_\ell)_{1\leq\ell\leq k}$, and $(z_{\alpha,i})_{i \ge 1}$ are independent copies of the Gaussian process with the covariance $\e z_{\alpha^1}z_{\alpha^2}= q_{\alpha^1\wedge \alpha^2}$, where $\alpha^1\wedge \alpha^2=\max\{\ell\geq 0\,:\, \alpha_j^1=\alpha^2_j \mbox{ for } j\leq \ell\}.$ The quantities (\ref{FEth}) and (\ref{FE}) are equal in this case, because, by (the proof of) \cite[Theorem~1.3]{SKmodel} and standard properties of the Ruelle probability cascades, see \cite[Theorem~4.4]{SKmodel}, both quantities are equal to the same continuous functional of the distribution of the array $(\mu^{-1}(\alpha^\ell\cdot\alpha^{\ell'}))_{\ell,\ell'\geq 1}$ under $\e \mfk R^{\otimes \infty}$ and correspondingly of the array $(q_{\alpha^\ell\wedge \alpha^{\ell'}})_{\ell,\ell'\geq 1}$ under $\e (\sum_{\al \in \N^k} v_\alpha \delta_\alpha)^{\otimes \infty}$; and these distributions are equal, due to the property of the Ruelle probability cascades that the distribution of an overlap array is determined by the distribution of one overlap. Moreover, denoting by $D \ge 0$ the smallest real number such that the support of $P_1$ is contained in $[-\sqrt{D}, \sqrt{D}]$, one can check (see for instance \cite{PM}, \cite[Proposition~2.1]{Mourrat19}, or Subsection~\ref{SecLipc} below) that for every $\mu,\nu \in \MM_b(\R_+)$, 
\begin{equation}
|F_N(t,\mu) - F_N(t,\nu)|\leq \frac{D}{2}\int_\R |\mu(s)-\nu(s)|\,\d s.
\label{FElip}
\end{equation}
In view of this, we can whenever convenient replace the measure $\mu$ by an atomic measure. Finally, using again standard properties of the Ruelle probability cascades (see e.g.\ \cite[Theorem 14.2.1]{SG} or \cite[Theorem~2.9]{SKmodel}), one can verify that $F_N(0,\mu)$ does not depend on $N$; we denote this quantity by
\begin{equation}  
\label{e.psi.f1}
\psi(\mu) := F_N(0,\mu) = F_1(0,\mu).
\end{equation}
We will recall a somewhat more explicit expression for $\psi(\mu)$ in (\ref{e.def.psi}) below. Denote by $U$ a uniform random variable over $[0,1]$, and for every probability measure $\mu$ on~$\R_+$, define $X_\mu := \mu^{-1}(U)$, 
where we recall that $\mu^{-1}$ is defined in \eqref{e.def.mu-1}. We also define, for every $s \in \R$,
\begin{equation*}  
\xi^*(s) := \sup_{r \ge 0} \Ll( rs - \xi(r) \Rr) .
\end{equation*}
Our first goal is to prove the following conjecture from \cite{Mourrat19} (specialized to the case where $P_N$ is a product measure). 
\begin{theorem}\label{t.main}
For every $t \ge 0$ and $\mu \in \mcl \MM_b(\R_+)$,
\begin{equation}  
\label{e.conj.h}
\lim_{N\to \infty}   F_N(t,\mu) 
= \inf_{\nu \in \mcl \MM_b(\R_+)} \Ll( \psi(\nu) + \frac{t}{2} \E   \xi^* \Ll( \frac{X_\nu - X_\mu}{t} \Rr)   \Rr) .
\end{equation}
\end{theorem}
The motivation in \cite{Mourrat19} for this statement is that the right side of (\ref{e.conj.h}), seen as a function of $(t,\mu)$, solves the formal Hamilton-Jacobi equation
\begin{equation}  
\label{e.hj}
\Ll\{
\begin{aligned}
& 2\dr_t f + \int\! \xi(-2\dr_\mu f) \, \d \mu = 0 & \quad \text{on } \R_+ \times \mcl M(\R_+), 
\\
& f(0,\,\cdot\,) = \psi  &\quad \text{on } \mcl M(\R_+).
\end{aligned}
\Rr.
\end{equation}
For discrete $\mu$ as in (\ref{e.music}), one can check that 
\begin{equation}
\label{e.error.term}
2 \dr_t F_N + \int \xi(-2\dr_\mu F_N) \, \d \mu = \e\la \xi(R_{1,2}) \ra_t
-\sum_{\ell=0}^{k} p_\ell\, \xi\Bigl(p_\ell^{-1}
\e \la R_{1,2}\1_{\{\alpha^1\wedge\alpha^2=\ell\}} \ra_t
\Bigr),
\end{equation}
where $R_{1,2} := N^{-1}\si^1 \cdot \si^2$ is the overlap of $\si^1$ and $\si^2$,
\begin{equation*} 
p_\ell:=\mu(\{q_\ell\})=\zeta_{\ell+1}-\zeta_\ell = \E \la \1_{\{\alpha^1\wedge\alpha^2=\ell\}} \ra
\end{equation*}
(see e.g.\ \cite[Lemma~2.3]{Mourrat19}), and $\la\,\cdot\,\ra$ denotes the average with respect to the Gibbs measure 
\begin{multline*}
\d G_N(\sigma,\alpha) \sim \exp \Big( \sqrt{t} H_N(\sigma) + \sum_{i=1}^{N}\sigma_i z_i^\mu(\alpha)
\\
- \frac{1}{2}Nt\xi(N^{-1} |\si|^2) - \frac{1}{2}\mu^{-1}(1)|\si|^2 \Big)  \,v_\alpha\, \d P_N(\sigma).
\end{multline*}
When $\xi$ is the square function, the right side of \eqref{e.error.term} can be interpreted as the conditional variance of the $\si$-overlap~$R_{1,2}$ given the $\al$-overlap $\alpha^1\cdot\alpha^2$. More generally, the right side of~\eqref{e.error.term} is small if and only if the conditional distribution of the overlap~$R_{1,2}$ given $\alpha^1\cdot\alpha^2$ is concentrated. This evokes the synchronization phenomenon used in the proof of the Parisi formula  by Talagrand in \cite{TPF}  along Guerra's interpolation~\cite{Guerra} with nearly optimal parameters; see also \cite{SG}. The idea of using the Hamilton-Jacobi techniques to study replica symmetric solution of the SK model was already utilized in \cite{GuerraSum}, and  one-step replica symmetry breaking analogues of the equation (\ref{e.error.term}) were derived and studied in various models in \cite{Barra, Agliari}.

\smallskip 

The main step in the proof of Theorem~\ref{t.main}, which is to pass to the limit $N \to \infty$ for the left side of \eqref{e.conj.h} and get some expression for the limit, is almost identical to the argument in \cite{SKvector} (specialized to the one-dimensional case), so we only outline the necessary modifications. The main tool is the synchronization mechanism developed in \cite{multiSK, SKpotts, SKvector} based on the overlap ultrametricity proved in~\cite{PUltra} for measures that satisfy the Ghirlanda-Guerra identities (and the fact that one has a lot of flexibility in enforcing these identities by way of small perturbations).  The synchronization has been applied in a variety of situations, e.g. \cite{JSK17, ConMin, Ko18}, and here we demonstrate another application. A particular synchronization that will be needed here is the one that forces the overlaps $\mu^{-1}(\alpha^1\cdot\alpha^2)$ and $R_{1,2}=N^{-1}\sigma^1\cdot \sigma^2$ to be deterministic functions of their sum in the thermodynamic limit. Notice that we need to use a synchronization argument here even in the case of Ising spins.

\smallskip

The reader may rightfully wonder what to make of the term $\xi(N^{-1} |\si|^2)$ appearing in the exponential in \eqref{FEth}, which was introduced for convenience but is otherwise a nuisance (except in the case of Ising spins, where it is deterministic and therefore causes no harm). The second goal of this paper is to explain how to remove this term and deduce from Theorem~\ref{t.main} the limit of the ``untampered'' free energy in \eqref{e.lim.free.energy}.
At present this is perhaps not as interseting as it sounds, since the proof of Theorem~\ref{t.main} could be modified to obtain the limit of the quantity without the term $\xi(N^{-1} |\si|^2)$ directly. However, it is likely that a more direct proof of Theorem~\ref{t.main} exists, in which case it is important to notice that Theorem~\ref{t.main} is indeed all the information needed to conclude. Moreover, we obtain in this way a somewhat different expression for the limit in \eqref{e.lim.free.energy} than that obtained in \cite{Pan05, SKvector}. 

\smallskip

In order to state this second result, we introduce two more parameters to the energy and write, for every $s,t \ge 0$, $\mu \in \MM_b(\R_+)$ and $h \in \R$,
\begin{multline}  
  \FF_N(s,t,\mu,h) :=  \frac 1 N \E \log \iint  \!\exp \Big( \sqrt{t} H_N(\sigma) + \sum_{i=1}^{N}\sigma_i z_i^\mu(\alpha)
\\
- \frac{1}{2}N(t-s)\xi(N^{-1} |\si|^2) - \frac{1}{2}\mu^{-1}(1)|\si|^2 + h |\si|^2\Big)  \, \d P_N(\sigma) \rpc(\alpha).
\end{multline}
Notice that when $s = 0$, this quantity is of the form covered by Theorem~\ref{t.main}, up to a redefinition of $P_N$ to absorb the term $\exp(h|\si|^2)$. We denote 
\begin{equation}  
\label{e.Psi.f1}
\Psi(\mu,h) := \FF_1(0,0,\mu,h) = \FF_N(0,0,\mu,h).
\end{equation}
 \begin{theorem}\label{t.mains}
For every $s,t \ge 0$, $\mu \in \mcl \MM_b(\R_+)$ and $h \in \R$, we have
\begin{equation*}  
\lim_{N\to \infty}   \FF_N(s,t,\mu,h) = \sup_{h'\in \R} \inf_{\nu \in \MM_b(\R_+)} \Ll( \Psi(\nu,h') +  \frac{t}{2} \E \xi^* \Ll( \frac{X_\nu - X_\mu}{t} \Rr) - \frac{s}{2}\xi^* \Ll( \frac{2(h'-h)}{s} \Rr)  \Rr) .
\end{equation*}
\end{theorem}
The intuition for this result is simple, and consists in writing the Hopf-Lax formula for the equation
\begin{equation}  
\label{e.approx.hj}
2\dr_s \FF_N - \xi(\dr_h \FF_N) \simeq 0.
\end{equation}
By setting $s = t = 1$, $\mu=\delta_0$, and $h=0$ in Theorem~\ref{t.mains}, we thus get the following new representation for the free energy of models with soft spins.
\begin{corollary} The limit free energy can be written as
\begin{multline}  
\lim_{N\to\infty} \frac 1 N \E \log \int  \!\exp \Ll(H_N(\sigma)\Rr)  \, \d P_N(\sigma)
\\
= \sup_{h\in \R} \inf_{\nu \in \MM_b(\R_+)} \Bigl( \Psi(\nu,h) +  \frac{1}{2}\int_{\R_+} \xi^*(r) \, \d \nu(r) - \frac{1}{2} \xi^*(2h) \Bigr) .
\end{multline}
\end{corollary}

\subsection*{Organization of the paper.}
In order to prove Theorem~\ref{t.main}, we first state a different expression for the left side of \eqref{e.conj.h} in Proposition \ref{p.Parisi} below. We then rewrite it in the form of the right side of \eqref{e.conj.h} in Section \ref{SecHLrepr}, by reasoning similarly to what was done in \cite{Mourrat19} in the case $\mu = \de_0$. We next turn to the proof of Proposition~\ref{p.Parisi} in Section~\ref{s.parisi}.  Finally, we provide the proof of Theorem~\ref{t.mains} in Section~\ref{s.mains}.

\section{Parisi formula}
In this section, we present the structure of the argument for identifying the limit on the left side of~\eqref{e.conj.h} in the more ``classical'' form in which Parisi formulas are usually stated. 
As a preparation for stating the formula we will obtain, we provide with an alternative description of the quantities $\psi$ and~$\Psi$ appearing in the main statements of Section~\ref{s.intro}. Given a probability measure $\nu$ on~$\R_+$, we write $\nu(s):=\nu([0,s])$. For every $\nu\in \MM_b(\Reals_+)$ and $\lambda \in \R$, we denote by $\Pnl = \Pnl(t,x)$ the solution of the equation
\begin{equation}
\Ll\{
\begin{aligned}
& \dr_t \Phi_{\nu,\lambda} = - \frac{1}{2}\bigl(\dr_x^2 \Pnl + \nu(t)(\dr_x \Pnl)^2\bigr)
&  \mbox{on } [0,\nu^{-1}(1)]\times \Reals,
\\
& \Pnl(\nu^{-1}(1),x)=\log\int_{\R} \exp\Ll(\sigma x + \lambda \sigma^2\Rr) \, \d P_1(\sigma) &  \text{for } x \in \R,
\end{aligned}
\Rr.
\label{ParisiPDE}
\end{equation}
and we set
\begin{equation}
\label{e.def.PP}
\PP(\nu,\lambda):=\Pnl(0,0).
\end{equation}
Using classical properties of Ruelle probability cascades, one can verify that the functions $\psi$ and $\Psi$ defined in \eqref{e.psi.f1} and \eqref{e.Psi.f1} respectively satisfy, for every $\mu \in \MM_b(\R_+)$ and $h \in \R$, 
\begin{equation}
\label{e.def.psi}
\Psi(\mu,h) = \PP \Ll( \mu, h - \frac 1 2 \mu^{-1}(1) \Rr) \quad \text{ and } \quad \psi(\mu) = \Psi(\mu,0).
\end{equation}
Given a probability measure $\zeta$ on $\Reals_+$, let $\zdm$ denote the probability measure on $\Reals_+$ whose cumulative distribution function satisfies
\begin{equation}
\zdm^{-1}(x):= \xi'(\zeta^{-1}(x))+\mu^{-1}(x).
\label{zdm}
\end{equation}
In other words, the c.d.f.\ of $\zdm$ is 
\begin{equation}
\zdm:= (\xi'\circ\zeta^{-1}+\mu^{-1})^{-1}.
\label{zdm2}
\end{equation}
Finally, let $\MM_{0,u}=\MM([0,u])$ denote the space of probability measures on $[0,u]$. 
\smallskip

Notice that it suffices to prove Theorem~\ref{t.main} for $t = 1$. Indeed, once the result is known in this case, we recover the general statement by replacing $\xi$ with $t \xi$. The main step towards the proof of Theorem~\ref{t.main} is the following result. For $\sigma_1$ is distributed according to $P_1$, we denote by $d$ and $D$ the smallest and largest points of the support of the distribution of $\sigma_1^2$. We also write, for every $r \in \R$,
\begin{equation*}  
\label{e.def.theta}
\theta(r):=r\xi'(r)-\xi(r).
\end{equation*}
\begin{proposition}
\label{p.Parisi}
For every $\mu \in \MM_b(\R_+)$, we have
\begin{multline}\label{Parisifix}
\lim_{N\to\infty} F_N(1,\mu) =
\sup_{u\in [d,D]}\inf_{\substack{\zeta\in \MM_{0,u} \\ \lambda\in\Reals}}\Bigl[
-\lambda u +
\PP\Bigl(\zdm,\lambda + \frac{\xi'(u)-\xi'(\zeta^{-1}(1))}{2}\Bigr)
\\
-\frac{1}{2}\int\limits_0^u\! \zeta(s)\, \d \theta(s)
- \frac{1}{2}\xi(u) - \frac{1}{2}\mu^{-1}(1)u 
\Bigr].
\end{multline}
\end{proposition}
We now outline the structure of the argument for obtaining Proposition~\ref{p.Parisi}. 
By the definitions of $d$ and $D$, when $\sigma\sim P_N=P_1^{\otimes N}$, we have that $N^{-1}|\sigma|^2\in [d,D]$, and any point $u\in [d,D]$ can be approximated by some $N^{-1}|\sigma|^2$ for large $N$ and $\sigma\in \supp P_N$. For every $u \in [d,D]$ and $\ep > 0$,  let
\begin{equation}
\Omega_N^\eps(u)=\Bigl\{\sigma\in \R^N \,:\, N^{-1}|\sigma|^2\in (u-\eps,u+\eps)\Bigr\},
\label{eqOmega}
\end{equation}
and consider
\begin{multline*}  
  F_N^\eps(u) := \frac 1 N \E \log \int\!\!\!\!\! \int\limits_{\Omega^\eps_N(u)} \exp \Big( H_N(\sigma) + \sum_{i=1}^{N}\sigma_i z_i^\mu(\alpha)
\\
- \frac{1}{2}N\xi(N^{-1} |\si|^2) - \frac{1}{2}\mu^{-1}(1)|\si|^2  \Big)\, \d P_N(\sigma)\, \rpc(\alpha).
\end{multline*}
The measure $\mu$ will be fixed throughout, so we keep the dependency of~$F_N^\ep(u)$ on $\mu$ implicit in the notation. 
It is clear that, denoting
\begin{equation}
p_N^\eps(u) := \frac 1 N \E \log \int\!\!\!\!\! \int\limits_{\Omega^\eps_N(u)}  \exp \Big( H_N(\sigma) + \sum_{i=1}^{N}\sigma_iz_i^\mu(\alpha) \Big)\, \d P_N(\sigma)\, \rpc(\alpha),
\label{pNeps}
\end{equation}
we have that as $\ep > 0$ tends to zero,
\begin{equation}  
F_N^\eps(u) :=  p_N^\eps(u)- \frac{1}{2}\xi(u) - \frac{1}{2}\mu^{-1}(1)u + \bigO{\eps}.
\label{Oeps}
\end{equation}
Proposition~\ref{p.Parisi} is a direct consequence of the following result.
\begin{theorem}\label{ThParisiU}
For every $u\in[d,D]$,
\begin{equation*}
\lim_{\eps\downarrow 0}\lim_{N\to\infty} p_N^\eps(u) =
\inf_{\substack{\zeta\in \MM_{0,u} \\ \lambda\in\Reals}}\Bigl[
-\lambda u +
\PP\Bigl(\zdm,\lambda + \frac{\xi'(u)-\xi'(\zeta^{-1}(1))}{2}\Bigr)
-\frac{1}{2}\int\limits_0^u\! \zeta(s)\, \d \theta(s)
\Bigr].
\end{equation*}
\end{theorem}
The proof of Theorem~\ref{ThParisiU} will be given in Section~\ref{s.parisi}. Before doing this, we show in the next section how to deduce Theorem~\ref{t.main} from Proposition~\ref{p.Parisi}. 

\section{Hopf-Lax representation}
\label{SecHLrepr}
In this section, we take the validity of Proposition~\ref{p.Parisi} for granted, and show that it implies Theorem~\ref{t.main}. We decompose the argument into five subsections.

\subsection{Change of variables} In (\ref{Parisifix}), let us make the change of variables
$$
\lambda\to \lambda  - \frac{\xi'(u)-\xi'(\zeta^{-1}(1))}{2}-\frac{1}{2}\zdm^{-1}(1)
=  \lambda  - \frac{1}{2}\xi'(u) -\frac{1}{2}\mu^{-1}(1),
$$
where the equality follows from (\ref{zdm}) with $x=1$. By the definition of $\Psi$ in \eqref{e.def.psi}, under this change of variables the second term in (\ref{Parisifix}) becomes 
$$
\PP\Bigl(\zdm,\lambda  -\frac{1}{2}\zdm^{-1}(1) \Bigr) = \Psi(\zdm, \lambda),
$$
and by cancelling and grouping other terms (recall that $\theta(u)=u\xi'(u)-\xi(u)$),
\begin{equation*}
\lim_{N\to\infty} F_N =
\sup_{u\in [d,D]}\inf_{\substack{\zeta\in \MM_{0,u} \\ \lambda\in\Reals}}\Bigl[
-\lambda u +
\Psi(\zdm,\lambda)
+ \frac{1}{2}\theta(u) -\frac{1}{2}\int\limits_0^u\! \zeta(s)\, \d \theta(s)
\Bigr],
\end{equation*}
where $F_N= F_N(1,\mu).$ An integration by parts gives that
\begin{equation*}
\lim_{N\to\infty} F_N =
\sup_{u\in [d,D]}\inf_{\substack{\zeta\in \MM_{0,u} \\ \lambda\in\Reals}}\Bigl[
-\lambda u +
\Psi(\zdm,\lambda )
+ \frac{1}{2}\int\limits_0^u\! \theta(s)\, \d \zeta(s)
\Bigr].
\end{equation*}
Since $\xi^*(\xi'(s))=\theta(s)$ for $s\geq 0$, if $U$ is a uniform random variable over $[0,1]$, we can write (again, recall (\ref{zdm}))
$$
\int\limits_0^u\! \theta(s)\, \d \zeta(s) = \int\limits_0^u\! \xi^*(\xi'(s))\, \d \zeta(s)
= \e \xi^*(\xi'(\zeta^{-1}(U)))
=  \e \xi^*(\zeta_\mu^{-1}(U)-\mu^{-1}(U)),
$$
and thus
\begin{equation*}
\lim_{N\to\infty} F_N =
\sup_{u\in [d,D]}\inf_{\substack{\zeta\in \MM_{0,u} \\ \lambda\in\Reals}}\Bigl[
-\lambda u +
\Psi(\zdm,\lambda )
+ \frac{1}{2} \e \xi^*(\zeta_\mu^{-1}(U)-\mu^{-1}(U))
\Bigr].
\end{equation*}
Notice that as $\zeta$ varies in $\in\MM_{0,u}$, the measures $\zeta_\mu$ defined in (\ref{zdm2}) span the set 
$$
\DD_u = \Bigl\{\nu \in \MM_b(\Reals_+)\,:\, \forall s, \nu(s)\leq \mu(s); \supp(\nu) \subseteq [0, \xi'(u)+\mu^{-1}(1)] \Bigr\},
$$
which means that, recalling that we write $X_\nu=\nu^{-1}(U)$,
\begin{equation}
\lim_{N\to\infty} F_N =
\sup_{u\in [d,D]}\inf_{\substack{\nu\in \DD_{u} \\ \lambda\in\Reals}}\Bigl[
-\lambda u +
\Psi(\nu,\lambda )
+ \frac{1}{2} \e \xi^*(X_\nu-X_\mu)
\Bigr].
\label{HLalmost}
\end{equation}

\subsection{Removing the constraint on the support}
Let us first show that we can remove the constraint $\supp(\nu) \subseteq [0, \xi'(u)+\mu^{-1}(1)]$ in $\DD_u$. without changing the value of the right side of \eqref{HLalmost}. For $\nu\in \MM_b(\Reals_+)$, let $\tilde\nu(s)=1$ for $s\geq \xi'(u)+\mu^{-1}(1)$ and $\tilde\nu(s)=\nu(s)$ otherwise. This corresponds to the truncation 
$$
X_{\tilde\nu}=\min(X_{\nu}, \xi'(u)+\mu^{-1}(1)).
$$ 
Let us show that
\begin{equation}
\xi^*(X_\nu-X_\mu) \geq \xi^*(X_{\tilde\nu}-X_\mu) + u\, |X_{\nu}-X_{\tilde\nu}|.
\label{starder}
\end{equation}
If $X_{\tilde\nu}=X_{\nu}$ then the two sides are equal. Otherwise, $X_{\nu}>X_{\tilde\nu}=\xi'(u)+\mu^{-1}(1)$. Since $\mu^{-1}(1)\geq X_\mu$, this implies that $X_{\nu}-X_\mu>X_{\tilde\nu}-X_\mu\geq \xi'(u)$. It remains to observe that ${\xi^*}'(s)={\xi'}^{-1}(s)\geq u$ if $s\geq \xi'(u)$, so (\ref{starder}) holds and
\begin{equation}
\e\xi^*(X_\nu-X_\mu) \geq \e\xi^*(X_{\tilde\nu}-X_\mu) + u\, \e|X_{\nu}-X_{\tilde\nu}|.
\end{equation}

On the other hand, if we define 
\begin{equation}
\Gamma_u(\nu):=\inf_{\lambda\in\Reals}\Bigl(-\lambda u + \Psi(\nu,\lambda)\Bigr),
\label{GammaDef}
\end{equation}
we will show below in Subsection \ref{SecLipc} below that
\begin{equation}
|\Gamma_u(\nu)-\Gamma_u(\tilde\nu)|\leq  \frac{u}{2} \e|X_\nu-X_{\tilde\nu}|.
\label{eqLipc}
\end{equation}
The last two inequalities imply that
$$
\Gamma_u(\tilde\nu) +  \frac{1}{2}\e \xi^*(X_{\tilde\nu}-X_\mu)\leq
\Gamma_u(\nu) +  \frac{1}{2}\e \xi^*(X_{\nu}-X_\mu),
$$
which means that the constraint  $\supp(\nu) \subseteq [0, \xi'(u)+\mu^{-1}(1)]$ can be removed and
\begin{equation}
\lim_{N\to\infty} F_N =
\sup_{u\in [d,D]}\inf_{\substack{\nu\in \DD_\mu \\ \lambda\in\Reals}}\Bigl[
-\lambda u +
\Psi(\nu,\lambda )
+ \frac{1}{2} \e \xi^*(X_\nu-X_\mu)
\Bigr],
\label{HLthere}
\end{equation} 
where $\DD_\mu := \bigl\{\nu \in \MM_b(\Reals_+)\,:\, \forall s, \nu(s)\leq \mu(s) \bigr\}.$

\subsection{Using convexity} 
Since $X_\nu-X_\mu\geq 0$ for $\nu\in \DD_\mu$, for such $\nu$ we have that $\e \xi^*(X_\nu-X_\mu)=\e \xi^{*}(|X_\nu-X_\mu|)$. Since $\xi^*(|s|)$ is convex and symmetric, $\e \xi^*(|X_\nu-X_\mu|)$ is a Wasserstein distance between $\mu$ and $\nu$ with the cost function $\xi^*(|x-y|)$ (see e.g. \cite[Theorem~2.18 and Remark~2.19(ii)]{villani}) and, therefore, $\e \xi^*(X_\nu-X_\mu)$ is convex in $\nu$ on $\DD_\mu$. Also, the Auffinger-Chen representation \cite{AC15} for the solution of equation (\ref{ParisiPDE}) (see also \cite{JT16}) implies that $\Psi(\nu,\lambda)$ is convex in $(\nu,\lambda)$. Therefore, by Sion's minimax theorem, see \cite[Corollary~3.3]{Sion}, we have
\begin{equation}
\lim_{N\to\infty} F_N =
\inf_{\substack{\nu\in \DD_\mu \\ \lambda\in\Reals}}\sup_{u\in [d,D]}
\Bigl[
-\lambda u +\Psi(\nu,\lambda )
+ \frac{1}{2} \e \xi^*(X_\nu-X_\mu)
\Bigr].
\label{HLthereSion}
\end{equation} 
Using that the boundary condition in \eqref{ParisiPDE} satisfies
\begin{equation*}
\log\int e^{\sigma x + \lambda \sigma^2} \, \d P_1(\sigma)
\geq
\log\int e^{\sigma x} \, \d P_1(\sigma) + \lambda d\1_{\{\lambda\geq 0\}}+\lambda D\1_{\{\lambda\leq 0\}},
\end{equation*}
we see that
$$
\sup_{u\in [d,D]}\Bigl[-\lambda u +\Psi(\nu,\lambda)\Bigr]
=
-\lambda d\1_{\{\lambda\geq 0\}}-\lambda D\1_{\{\lambda\leq 0\}} +\Psi(\nu,\lambda)\geq \Psi(\nu,0).
$$
This implies that infimum over $\lambda$ is achieved at $\lambda=0$ and,
recalling that $\psi(\nu)=\Psi(\nu,0)$,
\begin{equation}
\lim_{N\to\infty} F_N =
\inf_{\nu\in \DD_\mu}
\Bigl[\psi(\nu)+ \frac{1}{2} \e \xi^*(X_\nu-X_\mu)
\Bigr].
\label{HLDmu}
\end{equation}

\subsection{Removing the stochastic constraint.}
It remains to remove the constraint in $\DD_\mu,$ namely, $ \nu(s)\leq \mu(s)$. The reason we can do this is because, for arbitrary $\nu\in \MM_b(\Reals_+)$ and $\tilde\nu$ with the c.d.f. $\tilde\nu(s)=\min(\nu(s),\mu(s))$, we have
$$
\psi(\tilde\nu) \leq \psi(\nu) \mbox{ and }
\e \xi^*(X_{\tilde \nu}-X_\mu)  = \e \xi^*(X_\nu-X_\mu).
$$ 
The second equality holds because $\xi^*(s)=0$ for $s\leq 0$ and $X_{\tilde\nu}=\max(X_\nu,X_\mu)$. The first inequality follows from the monotonicity of $\psi$ in $\nu$, which can be seen as follows. First of all, in the definition of $\Psi(\nu,h)$ in \eqref{e.def.psi}, if we take any $a\geq \nu^{-1}(1)$ and let $\Phi^a(t,x)$ be the solution of the equation
\begin{equation}
\Ll\{
\begin{aligned}
& \dr_t \Phi^a = - \frac{1}{2}\bigl(\dr_x^2 \Phi^a + \mu(t)(\dr_x \Phi^a)^2\bigr)
&  \mbox{on } [0,a]\times \Reals,
\\
& \Phi^a(a,x)=\log\int_{\R} \exp\Ll(\sigma x + \lambda \sigma^2\Rr) \, \d P_1(\sigma) &  \text{for } x \in \R,
\end{aligned}
\Rr.
\label{ParisiPDEa}
\end{equation}
and define $\PP^a(\nu,\lambda):= \Phi^a(0,0)$, then
\begin{align}
\Psi(\nu,h)&=\PP\Bigl(\nu,h-\frac{1}{2}\nu^{-1}(1)\Bigr)
=\PP^a\Bigl(\nu,h-\frac{1}{2}a\Bigr).
\label{psidef2}
\end{align}
This means that when we compare $\Psi(\nu,\lambda)$ and $\Psi(\nu',\lambda)$, we can solve the PDE on the same interval that includes the support of both measures. Since the solution is monotone in the c.d.f. $\nu(s)$, monotonicity of the mapping $\nu \mapsto \Psi(\nu,\lambda)$ follows. This proves
\begin{equation}
\lim_{N\to\infty} F_N =
\inf_{\nu\in \MM_b(\R_+)}
\Bigl[\psi(\nu)+ \frac{1}{2} \e \xi^*(X_\nu-X_\mu)
\Bigr].
\label{HLDM}
\end{equation}
Finally, rescaling $\xi\to t\xi$ and recalling the definition of $\psi$ in \eqref{e.def.psi}, we get
\begin{align}
\lim_{N\to\infty} F_N(t,\mu) 
&= \inf_{\nu\in\MM_b(\Reals_+)}\Bigl[ \psi(\nu)+\frac{t}{2}\e \xi^*\Bigl(\frac{X_\nu-X_\mu}{t}\Bigr)
\Bigr].
\end{align}
This finishes the proof that Proposition~\ref{p.Parisi} implies Theorem \ref{t.main}, up to the verification of (\ref{eqLipc}).

\subsection{Lipschitz continuity}\label{SecLipc}

We now show (\ref{eqLipc}). Let us recall the definition of the set in (\ref{eqOmega}) and define
\begin{align*}
f_N(\nu,\eps)&=
\frac{1}{N}\e \log  \int\!\!\!\!\! \int\limits_{\Omega^\eps_N(u)}  \exp\Bigl(\sum_{i=1}^{N}\sigma_iz_i^\nu(\alpha)-\frac{1}{2}\nu^{-1}(1)|\sigma|^2
\Bigr)\, \d P_N(\sigma) \rpc(\alpha).
\end{align*}
Let us first suppose that $\nu$ is discrete. If we recall (\ref{GammaDef}), the results of \cite[Section~7]{SKvector} (specialized to the one-dimensional case) show that, for discrete $\nu$,
\begin{equation}
\Gamma_u(\nu)=\lim_{\eps\downarrow 0} \lim_{N\to\infty} f_N(\nu,\eps).
\label{Sec7SKvector}
\end{equation}
Given $\nu,\tilde\nu\in\MM_b(\Reals_+)$, we can interpolate between $f_N(\nu,\eps)$ and $f_N(\tilde\nu,\eps)$ by replacing $z_i^\nu(\alpha)$ by $\sqrt{t}z_i^\nu(\alpha)+\sqrt{1-t}z_i^{\tilde\nu}(\alpha)$ with covariance $(t\nu^{-1}+(1-t)\tilde\nu^{\,-1})(\alpha^1\cdot\alpha^2)$ and replacing $\nu^{-1}(1)$ by $t\nu^{-1}(1)+(1-t)\tilde\nu^{\,-1}(1)$. Then the derivative of $f_N(\nu,\eps)$ in $t $ along this interpolation path equals
$$
-\frac{1}{2}\e\la R_{1,2}(\nu^{-1}(\alpha^1\cdot\alpha^2)-\tilde\nu^{\,-1}(\alpha^1\cdot\alpha^2)) \ra,
$$
where $\la\,\cdot\,\ra$ is the average with respect to the Gibbs measure 
\begin{multline*}
\d G_N(\sigma,\alpha) 
\sim \exp\Bigl(\sum_{i=1}^{N}\sigma_i\bigl(\sqrt{t}z_i^\nu(\alpha)+\sqrt{1-t}z_i^{\tilde\nu}(\alpha)\bigr)
\\
-\frac{1}{2}\bigl(t\nu^{-1}(1)+(1-t)\tilde\nu^{\,-1}(1)\bigr) |\sigma|^2
\Bigr)\, \d P_N(\sigma) \rpc(\alpha)
\end{multline*}
on $ \Omega^\eps_N(u)\times \supp(\mfk R)$. Since, by Cauchy's inequality, $|R_{1,2}|\leq u+\eps$ for $\sigma^1,\sigma^2\in \Omega^\eps_N(u)$, the above derivative is bounded by
$$
\frac{u+\eps}{2}\e\la \bigl|\nu^{-1}(\alpha^1\cdot\alpha^2)-\tilde\nu^{\,-1}(\alpha^1\cdot\alpha^2)\bigr| \ra
=
\frac{u+\eps}{2}\e |X_\nu-X_{\tilde\nu}|,
$$
where we also used the fact that the distribution of $\alpha^1\cdot\alpha^2\sim U[0,1]$ under $\e G_N^{\otimes 2}$ is the same as under $\e\mathfrak{R}^{\otimes 2}$ by the properties of the Ruelle probability cascades (see e.g. \cite[Theorem 4.4]{SKmodel}). This and (\ref{Sec7SKvector}) imply (\ref{eqLipc}) for discrete $\nu$, and by extension for all $\nu\in\MM_b(\Reals_+)$.

\section{Proof of the Parisi formula}
\label{s.parisi}

The goal of this section is to prove Theorem \ref{ThParisiU}, which we recall implies Proposition~\ref{p.Parisi}. We first prove the upper and then the lower bound.

\subsection{Upper bound}
The upper bound is proved by the standard Guerra replica-symmetry-breaking interpolation \cite{Guerra}. By Lipschitz continuity (\ref{FElip}), it is enough to consider discrete $\mu$ and suppose that the infimum in Theorem \ref{ThParisiU} is taken also over discrete distributions $\zeta\in \MM_{0,u}$ such that $\zeta^{-1}(1)=u.$ Given such $\zeta$, let $z(\alpha)$ and $y(\alpha)$ be independent Gaussian processes (conditionally on $\mfk R$) indexed by $\alpha\in\supp(\mfk R)$ with covariances
$$
\e z(\alpha^1)z(\alpha^2)=\xi'(\zeta^{-1}(\alpha^1\cdot\alpha^2)),\,\,
\e y(\alpha^1)y(\alpha^2)=\theta(\zeta^{-1}(\alpha^1\cdot\alpha^2)),
$$ 
and let $z_i(\alpha)$ be independent copies of $z(\alpha)$ for $i\geq 1$. We assume these processes to also be independent of $H_N$ and $z_i^\mu(\alpha)$, conditionally on $\mfk R$. Consider an interpolating free energy, for $t\in [0,1],$
\begin{equation}
\varphi(t) := \frac 1 N \E \log \int\!\!\!\!\! \int\limits_{\Omega^\eps_N(u)}  \exp (H_{N,t}(\sigma,\alpha))\, \d P_N(\sigma)\, \rpc(\alpha).
\end{equation}
where the interpolation Hamiltonian is defined by
$$
H_{N,t}(\sigma,\alpha) := \sqrt{t}H_N(\sigma)  + \sqrt{1-t}\sum_{i=1}^{N}\sigma_iz_i(\alpha)+\sqrt{t}\sqrt{N}y(\alpha) + \sum_{i=1}^{N}\sigma_iz_i^\mu(\alpha).
$$
One can see that
\begin{multline*}
\frac{2}{N}\e \, \dr_t H_{N,t}(\sigma^1,\alpha^1) \, H_{N,t}(\sigma^2,\alpha^2)
=
\Delta(R_{1,2},\alpha^1\cdot\alpha^2)
\\
:=
\xi(R_{1,2})-R_{1,2}\xi'(\zeta^{-1}(\alpha^1\cdot\alpha^2))+\theta(\zeta^{-1}(\alpha^1\cdot\alpha^2)).
\end{multline*}
Since $R_{1,1}=N^{-1}|\sigma^1|^2\in (u-\eps,u+\eps)$ whenever $\sigma^1\in \Omega^\eps_N(u)$ and we also assumed that $\zeta^{-1}(\alpha^1\cdot \alpha^1)=\zeta^{-1}(1)=u,$ we have $|\Delta(R_{1,1},\alpha^1\cdot\alpha^1)|=\bigO\eps.$ Therefore, by the usual Gaussian integration by parts,
$$
\varphi'(t)=-\frac{1}{2}\e\la \Delta(R_{1,2},\alpha^1\cdot\alpha^2) \ra + \bigO\eps,
$$
where $\la\,\cdot\,\ra$ is the average with respect to the Gibbs measure 
$$
\d G_N(\sigma,\alpha) 
\sim \exp( H_{N,t}(\sigma,\alpha))\, \d P_N(\sigma) \rpc(\alpha)
$$
on $ \Omega^\eps_N(u)\times \supp(\mfk R)$. When $\xi$ is convex, $\xi(a)-a\xi'(b)+\theta(b)\geq 0$, which proves that $\varphi(1)\leq \varphi(0) + \bigO\eps.$ When $\xi$ is only convex on $\Reals_+$, one can add a small perturbation that enforces the Ghirlanda-Guerra identities and, as a result, enforces asymptotic positivity of $R_{1,2}$ (see \cite{T2} or \cite[Chapter~3]{SKmodel}). 

We can bound $\varphi(0)$ from above by adding $-N\lambda u+ \lambda |\sigma|^2$ to the Hamiltonian, which on $\Omega^\eps_N(u)$ is bounded in absolute value by $N|\lambda |\eps$ and, as a result,
\begin{align*}
\varphi(0)
&\leq |\lambda |\eps
-\lambda u + \frac 1 N \E \log\iint\! \exp\Bigl(
\sum_{i=1}^{N}\sigma_i \bigl(z_i(\alpha)+z_i^\mu(\alpha)\bigr)
+\lambda |\sigma|^2
\Bigr)\, \d P_N(\sigma) \rpc(\alpha)
\\
&=
 |\lambda |\eps -\lambda u +\E \log\iint\! \exp\Bigl(
\sigma_1\bigl(z_1(\alpha)+z_1^\mu(\alpha)\bigr)
+\lambda \sigma_1^2
\Bigr)\, \d P_1(\sigma_1) \rpc(\alpha)
\\
&=
 |\lambda |\eps -\lambda u + \PP\bigl(\zeta_\nu,\lambda\bigr),
\end{align*}
by the standard properties of the Ruelle probability cascades and the fact that $z_1(\alpha)+z_1^\mu(\alpha)$ has covariance $\zdm^{-1}(\alpha^1\cdot\alpha^2)$
, where $\zdm(s)$ was defined in (\ref{zdm}), (\ref{zdm2}). On the other hand, again, by the standard properties of the Ruelle probability cascades (recall the notation in (\ref{pNeps})),
\begin{align*}
\varphi(1) &=p_N^\eps(u)+ 
\frac 1 N \E \log \int\! \exp( \sqrt{N}y(\alpha))\,\rpc(\alpha)
\\
&=
p_N^\eps(u)+ 
\E \log \int\! \exp( y(\alpha))\,\rpc(\alpha)
=
p_N^\eps(u)+ \frac{1}{2}\int\limits_0^u\! \zeta(s)\, \d \theta(s).
\end{align*}
Putting everything together shows that 
\begin{equation}
\lim_{\eps\downarrow 0}\lim_{N\to\infty} p_N^\eps(u) \leq
-\lambda u + \PP(\zdm,\lambda)-\frac{1}{2}\int\limits_0^u\! \zeta(s)\, \d \theta(s),
\end{equation}
for all discrete distributions $\zeta\in \MM_{0,u}$ such that $\zeta^{-1}(1)=u.$ Since continuous extension of $\PP(\zdm,\lambda)$ to all $\zeta\in \MM_{0,u}$ not necessarily satisfying $\zeta^{-1}(1)=u$ is exactly 
\begin{equation}
\PP\Bigl(\zdm,\lambda + \frac{\xi'(u)-\xi'(\zeta^{-1}(1))}{2}\Bigr)
\label{PPcext}
\end{equation}
(this is analogous to why the term $-\frac{1}{2}\mu^{-1}(1)|\sigma|^2$ was included in the definition of $F_N(t,\mu)$), this finishes the proof of the upper bound.

\subsection{Lower bound} The proof of the lower bound is identical to the one-dimensional case of \cite{SKvector}, with some simplifications due to the one-dimensional nature of our problem and one minor modification to account for the presence of the term $\sum_{i=1}^{N}\sigma_iz_i^\mu(\alpha)$ that we will now explain. 

The main effect of this term is that the cavity fields (in the first term) of the Aizenman-Sims-Starr representation will be of the form $c_i(\sigma,\alpha):=z_{i}(\sigma)+z_i^\mu(\alpha)$ for $(\sigma,\alpha)\in \Omega^\eps_N(u)\times \supp(\mfk R)$ with covariance
\begin{equation}
C_{\ell,\ell'}:=\e c_i(\sigma^\ell,\alpha^\ell)c_i(\sigma^{\ell'},\alpha^{\ell'}) = \xi'(R_{\ell,\ell'})+\mu^{-1}(\alpha^\ell\cdot\alpha^{\ell'}).
\end{equation}
To understand the distribution of the array $(C_{\ell,\ell'})_{\ell,\ell'\geq 1}$ under the Gibbs measure that arises in the cavity computation, we can use the synchronization mechanism from \cite{multiSK}  to synchronize the overlaps $R_{1,2}$ and $q_{1,2}:=\mu^{-1}(\alpha^1\cdot\alpha^{2})$. This can be done by including terms in the perturbation Hamiltonian with covariances given by monomials $R_{1,2}^n q_{1,2}^m$ and then use Theorem 4 in \cite{multiSK} to show that both $R_{1,2}$ and $q_{1,2}$ are non-decreasing $1$-Lipschitz functions of their sum in the thermodynamic limit. 

If we think of the sum $S_{1,2}:=R_{1,2}+q_{1,2}$ as the quantile transform $\nu^{-1}(U)$ of $\nu=\LL(S_{1,2})$ and uniform $U\sim U[0,1]$, then both $R_{1,2}$ and $q_{1,2}$ are non-decreasing functions of $U$, which means they must be quantile transforms of their distributions. The distribution of $q_{1,2}$ is $\mu$ for all $N$ by the properties of the Ruelle probability cascades (\cite[Theorem 4.4]{SKmodel}) and, thus, in the limit. If the limiting distribution of $R_{1,2}$ (as usual, along some subsequence) is $\zeta$ then (recalling (\ref{zdm}))
$$
\xi'(R_{\ell,\ell'})+\mu^{-1}(\alpha^\ell\cdot\alpha^{\ell'})
\stackrel{d}{=} 
\xi'(\zeta^{-1}(U))+\mu^{-1}(U) = \zeta_\mu^{-1}(U).
$$
This means that $C_{\ell,\ell'}\sim \zeta_\mu$. Similarly, the cavity fields $y(\sigma)$ coming from the Onsager correction in the second term in the Aizenman-Sims-Starr scheme will have covariance
$$
\theta(R_{\ell,\ell'})\stackrel{d}{=} \theta(\zeta^{-1}(U)),
$$
in the thermodynamic limit. If $\zeta^{-1}(1)=u$ then the lower bound one obtains by the cavity computation is equal to
\begin{equation}
\inf_{\lambda\in\Reals}\Bigl[
-\lambda u +\PP(\zdm,\lambda)-\frac{1}{2}\int\limits_0^u\! \zeta(s)\, \d \theta(s)
\Bigr].
\end{equation}
For general $\zeta\in \MM_{0,u}$, we again appeal to the fact that (\ref{PPcext}) is a continuous extension from general $\zeta$ of $\PP(\zdm,\lambda)$ for $\zeta$ satisfying $\zeta^{-1}(1)=u.$

\section{Proof of Theorem \ref{t.mains}}
\label{s.mains}
The goal of this section is to prove Theorem~\ref{t.mains}. We obtain this result by combining Theorem~\ref{t.main} with the observation in \eqref{e.approx.hj} that $\FF_N$ satisfies a Hamilton-Jacobi equation, up to a small error term. Denote by $\la \,\cdot\, \ra$ the Gibbs measure
\begin{multline*}  
\d G_N(\sigma,\alpha) 
\sim 
\exp \Big( \sqrt{t} H_N(\sigma) + \sum_{i=1}^{N}\sigma_i z_i^\mu(\alpha)
\\
- \frac{1}{2}N(t-s)\xi(N^{-1} |\si|^2) - \frac{1}{2}\mu^{-1}(1)|\si|^2 + h |\si|^2\Big)  \, \d P_N(\sigma) \rpc(\alpha).
\end{multline*}
Similarly to the observations in \cite[Section~1]{HJinfer} concerning the Curie-Weiss model (with $\xi$ replaced by the square function there), we have
\begin{equation*}  
\dr_s \FF_N = \frac 1 2 \E \la \xi \Ll( N^{-1} |\si|^2 \Rr) \ra,
\end{equation*}
while
\begin{equation}  
\label{e.drh.FN}
\dr_h \FF_N = \E \la N^{-1} |\si|^2 \ra,
\end{equation}
and
\begin{equation}  
\label{e.drh2.FN}
\dr_h^2 \FF_N = N \, \E \la \Ll(N^{-1} |\si|^2  - \E \la N^{-1} |\si|^2 \ra \Rr)^2 \ra.
\end{equation}
Notice in particular that $\dr_h \FF_N \ge 0$, $\dr_h^2 \FF_N \ge 0$, and since the support of the measure $P_1$ is assumed to be bounded, the derivatives $\dr_s \FF_N$ and $\dr_h \FF_N$ are bounded uniformly in~$N$. Moreover, since $\xi$ is locally Lispschitz continuous, there exists a constant $C < \infty$ such that, for every~$N$,
\begin{align}  
\label{e.estim.approx.hj}
\Ll|2 \dr_s \FF_N - \xi(\dr_h \FF_N)\Rr| 
& = \Ll| \E \la \xi \Ll( N^{-1} |\si|^2 \Rr) \ra - \xi \Ll( \E \la N^{-1} |\si|^2 \ra \Rr) \Rr| 
\\
\notag
& \le C \E \la \Ll| ( N^{-1} |\si|^2 - \E \la N^{-1} |\si|^2 \ra  \Rr| \ra
\\
\notag
& \le C (N^{-1} \dr_h^2 \FF_N)^\frac 1 2 .
\end{align}
We fix $t \ge 0$ and $\mu \in \MM_b(\R_+)$, and denote by $f = f(s,h) : \R_+ \times \R \to \R$ the candidate limit for $\FF_N$, namely
\begin{equation}  
\label{e.def.f}
f(s,h) := \sup_{h'\in \R} \Ll( \hat \Psi(h')  - \frac{s}{2}\xi^* \Ll( \frac{2(h'-h)}{s} \Rr) \Rr) ,
\end{equation}
where we set
\begin{equation*}  
\hat \Psi(h') :=  \inf_{\nu \in \MM_b(\R_+)} \Ll( \Psi(\nu,h') +  \frac{t}{2} \E \xi^* \Ll( \frac{X_\nu - X_\mu}{t} \Rr) \Rr).
\end{equation*}
Notice that we do not display the dependency of $f$ and $\hat \Psi$ on $t \ge 0$ and $\mu \in \MM_b(\R_+)$; we allow ourselves to do this since these parameters will be kept fixed throughout the section. For the same reason, from now on, \emph{we write $\FF_N(s,h)$ in place of $\FF_N(s,t,\mu,h)$.}

\smallskip

Recalling that, by Theorem~\ref{t.main}, the quantity $\hat \Psi(h')$ is the limit of $\FF_N(0,h')$, and using \eqref{e.drh.FN} and \eqref{e.drh2.FN}, it is clear that $\hat \Psi$ is uniformly Lipschitz continuous, nondecreasing, and convex. One can check that these properties transfer to the function $f$: it is uniformly Lipschitz continuous over $\R_+ \times \R$, and for each fixed $s \ge 0$, the mapping $h \mapsto f(s,h)$ is nondecreasing and convex (see for instance \cite[Lemmas~I.3.3.2 and I.3.3.3]{evans}). In particular, by the Rademacher theorem, the function~$f$ is differentiable almost everywhere. Moreover, the expression for $f$ in \eqref{e.def.f} is a Hopf-Lax formula; as a consequence, see \cite[Theorem~I.3.3.5]{evans}, for every $(s,h) \in (0,\infty) \times \R$, if $f$ is differentiable at $(s,h)$, then
\begin{equation}
\label{e.exact.hj}
2 \dr_s f(s,h) - \xi(\dr_h f(s,h)) = 0.
\end{equation}

\smallskip

Our goal is to show that $\FF_N(s,h)$ converges to $f(s,h)$. While we refrain  from writing down a general statement, we list here all the properties of these functions that will be used below: 
\begin{enumerate}
\item the functions are uniformly Lipschitz, with a common Lipschitz constant; \item the functions are nondecreasing and convex in $h$; 
\item for each $h$, we have $\lim_{N \to \infty} \FF_N(0,h) = f(0,h)$;
\item the function $f$ satisfies the equation \eqref{e.exact.hj} almost everywhere, while the function $\FF_N$ satisfies the same equation, up to an error that we will show to be small after integration in $h$, uniformly over $s$. 
\end{enumerate}

\begin{proof}[Proof of Theorem~\ref{t.mains}]
We split the proof into two steps. 

\smallskip

\emph{Step 1.} We write down an equation for the difference between $\FF_N$ and $f$ and state some elementary bounds. We denote
\begin{equation*}  
w_N := \FF_N - f \quad \text{and} \quad \err := 2 \dr_s \FF_N - \xi(\dr_h \FF_N),
\end{equation*}
so that, almost everywhere in $\R_+ \times \R$,
\begin{align*}  
2 \dr_s w_N 
& = \xi(\dr_h \FF_N) - \xi(\dr_h f) + \err
\\
\notag
& = \int_0^1 \dr_u \Ll( \xi(u\dr_h \FF_N + (1-u) \dr_h f) \Rr) \, \d u + \err
\\
\notag
& = \b_N \dr_h w_N + \err,
\end{align*}
where we have set
\begin{equation*}  
\b_N := \int_0^1 \xi'(u \dr_h \FF_N + (1-u) \dr_h f)  \, \d u.
\end{equation*}
Let $\phi \in C^\infty(\R)$ be a smooth function such that $\phi(0) = 0$ and $|\phi'| \le 1$, and define $v_N := \phi(w_N)$. By the chain rule, we have
\begin{equation}  
\label{e.equation.vN}
2\dr_s v_N - \b_N \dr_h v_N = \phi'(w_N) \err \qquad \text{a.e. in } \R_+ \times \R.
\end{equation}
It will be convenient to be allowed to differentiate $\b_N$ in $h$. In order to make this rigorous, we regularize $\b_N$ a bit, by convolution with a smooth kernel. Let $\zeta \in C^\infty_c(\R)$ be a smooth function with compact support such that $\int_\R \zeta = 1$, and for each $\ep > 0$, $s \ge 0$ and $h \in \R$, denote 
\begin{equation*}  
f_\ep (s,h) := \ep^{-1} \int_{\R} f(s,h-h') \zeta(\ep^{-1} h') \, \d h',
\end{equation*}
and
\begin{equation*}  
\b_{N,\ep} := \int_0^1 \xi'(u\dr_h \FF_N + (1-u) \dr_h f_\ep) \, \d u.
\end{equation*}
One can check that for each fixed $N$ and $s \ge 0$, the function $\b_{N,\ep}(s,\cdot)$ converges to $\b_N(s,\cdot)$ almost everywhere in $\R$ as $\ep$ tends to zero (see for instance \cite[Theorem~C.5.7]{evans}). Moreover, 
\begin{equation*}  
\dr_h \b_{N,\ep} = \int_0^1 (u \dr_h^2 \FF_N + (1-u) \dr_h^2 f_\ep) \xi''(u\dr_h \FF_N + (1-u) \dr_h f_\ep) \, \d u,
\end{equation*}
and since $\dr_h \FF_N$, $\dr_h f_\ep$, $\dr_h^2 \FF_N$, and $\dr_h^2 f_\ep$, are all nonnegative, and $\xi''$ maps $\R_+$ to $\R_+$, we deduce that
\begin{equation}
\label{e.div.b}
\dr_h \b_{N,\ep} \ge 0.
\end{equation}
Notice also that, since $\FF_N$ and $f$ are Lipschitz with a common Lipschitz constant, we have that $\|\b_{N,\ep}\|_{L^\infty}$ is bounded uniformly over $N$ and $\ep$. We write
\begin{equation*}  
R := 1+\sup_{N,\ep} \|\b_{N,\ep}\|_{L^\infty} < \infty.
\end{equation*}

\smallskip

\emph{Step 2.} We fix $S\ge 1$ for the remainder of the proof, and study the quantity
\begin{equation*}  
J_N(s) := \int_{-R(S-s)}^{R(S-s)} v_N(s,h) \, \d h, \qquad s \in \Ll[ 0,S \Rr] .
\end{equation*}
The function $J_N$ is Lipschitz continuous, and for almost every $s \in [0,S]$,
\begin{equation*}  
\dr_s J_N(s) = \int_{-R(S-s)}^{R(S-s)} \dr_s v_N(s,h) \, \d h - Rv_N(s,R(S-s)) -Rv_N(s,-R(S-s)).
\end{equation*}
By \eqref{e.equation.vN}, we also have
\begin{equation*}  
2\int_{-R(S-s)}^{R(S-s)} \dr_s v_N \, \d h 
= \int_{-R(S-s)}^{R(S-s)} \Ll(\b_{N,\ep} \dr_h v_N + (\b_N - \b_{N,\ep})\dr_h v_N + \phi'(w_N) \err \Rr) \, \d h,
\end{equation*}
where we kept it implicit in the notation that the functions in the integrands are evaluated at $(s,h)$.
We now estimate the contribution of each term on the right side in turn. By the definition of $R$ and an integration by parts, we have
\begin{equation*}  
\Ll| \int_{-R(S-s)}^{R(S-s)} \Ll(\b_{N,\ep} \dr_h v_N +v_N \dr_h \b_{N,\ep} \Rr) \, \d h \Rr| \le Rv_N(s,R(S-s)) +Rv_N(s,-R(S-s))
\end{equation*}
(recall that $v_N \ge 0$). Using also \eqref{e.div.b}, we deduce that
\begin{equation*}  
\int_{-R(S-s)}^{R(S-s)} \b_{N,\ep} \dr_h v_N \, \d h - Rv_N(s,R(S-s)) -Rv_N(s,-R(S-s)) \le 0.
\end{equation*}
Recalling also that for each fixed $N$, we have that $\b_{N,\ep}(s\cdot)$ converges to $\b_N(s,\cdot)$ almost everywhere, and using the dominated convergence theorem, we see that
\begin{equation*}  
\lim_{\ep \to 0} \int_{-R(S-s)}^{R(S-s)}  (\b_N - \b_{N,\ep})\dr_h v_N \, \d h = 0.
\end{equation*}
Summarizing, we have shown that, for almost every $s \in [0,S]$, 
\begin{equation*}  
\dr_s J_N(s) \le \frac 1 2 \int_{-R(S-s)}^{R(S-s)} \phi'(w_N) \err \, \d h.
\end{equation*}
Recalling that $|\phi'| \le 1$ and using \eqref{e.estim.approx.hj}, we deduce that
\begin{equation*}  
\dr_s J_N(s) \le C N^{-\frac 1 2} \int_{-R(S-s)}^{R(S-s)} \Ll(\dr_h^2 \FF_N\Rr)^\frac 1 2 \, \d h.
\end{equation*}
Allowing the constant $C$ to depend on $R$ and $S$, we can use Jensen's inequality to deduce that, for almost every $s \in [0,S]$, 
\begin{align*}  
\dr_s J_N(s) \le C N^{-\frac 1 2} \Ll(\int_{-R(S-s)}^{R(S-s)} \dr_h^2 \FF_N \, \d h\Rr)^\frac 1 2,
\end{align*}
and since $\FF_N$ is Lipschitz uniformly over $N$, the integral on the right side is bounded. To sum up, we have thus shown that for every $s \in [0,S]$,
\begin{align*}  
J_N(s) = J_N(0) + \int_0^s \dr_s J_N(r) \, \d r
\le J_N(0) + C s N^{-\frac 1 2}.
\end{align*}
Recalling the definition of $v_N$, fixing $s = \frac S 2$, this implies in particular, up to a redefinition of $C < \infty$,
\begin{equation*}  
\int_{-\frac{RS}{2}}^{\frac{RS}{2}} \phi(\FF_N-f)(s,h) \, \d h \le \int_{-\frac{RS}{2}}^{\frac{RS}{2}} \phi(\FF_N-f)(0,h) \, \d h + C N^{-\frac 1 2}.
\end{equation*}
Notice also that the constant $C < \infty$ does not depend on our choice of function $\phi$ such that $|\phi'| \le 1$. We thus deduce that 
\begin{equation*}  
\int_{-\frac{RS}{2}}^{\frac{RS}{2}} \Ll|\FF_N-f\Rr|(s,h) \, \d h \le \int_{-\frac{RS}{2}}^{\frac{RS}{2}} \Ll|\FF_N-f\Rr|(0,h) \, \d h + C N^{-\frac 1 2}.
\end{equation*}
Finally, by the dominated convergence theorem, the integral on the right side converges to $0$ as $N$ tends to infinity. We have thus shown that
\begin{equation*}  
\lim_{N \to \infty} \int_{-\frac{RS}{2}}^{\frac{RS}{2}} \Ll|\FF_N-f\Rr|(s,h) \, \d h = 0.
\end{equation*}
Recall that $R > 0$ and that our choice of $S \ge 1$ was arbitrary. To conclude for the pointwise convergence of $\FF_N$ to $f$, it suffices to use the Lipschitz regularity of $\FF_N$. Explicitly, for every $\ep > 0$, we can write
\begin{equation*}  
\FF_N(s,h) = \FF_N(s,h) - \frac{1}{2\ep} \int_{-\ep}^\ep \FF_N(s,h') \, \d h' + \frac 1 {2\ep} \int_{-\ep}^\ep\FF_N(s,h') \, \d h',
\end{equation*}
and we have seen above that the last integral converges to the corresponding integral with $\FF_N$ replaced by $f$ as $N$ tends to infinity. Moreover, by the Lipschitz continuity of $\FF_N$,
\begin{equation*}  
\Ll| \FF_N(s,h) - \frac{1}{2\ep} \int_{-\ep}^\ep \FF_N(s,h') \, \d h' \Rr| \le C \ep.
\end{equation*}
Hence, sending $N$ to infinity first and then $\ep$ to zero allows us to conclude that for each $s \ge 0$ and $h \in \R$, we have indeed $\lim_{N \to \infty} \FF_N(s,h) = f(s,h)$.
\end{proof}

\noindent \textbf{Acknowledgements.} JCM was partially supported by the ANR grants LSD (ANR-15-CE40-0020-03) and Malin (ANR-16-CE93-0003). DP was partially supported by NSERC.

\small

\end{document}